\documentclass[a4paper,12pt]{amsart}
\usepackage[english]{babel}
\usepackage{amsmath,amsthm,amssymb,amsfonts}
\usepackage{multicol}
\usepackage[titletoc]{appendix}
\usepackage{soul}

\usepackage[pdftex]{color}
\usepackage[bookmarks=true,hyperindex,pdftex,colorlinks, citecolor=blue,linkcolor=blue, urlcolor=blue]{hyperref}
\usepackage[dvipsnames]{xcolor}
\usepackage{mathtools}
\usepackage{graphicx}
\usepackage{pgf}
\usepackage{tikz-cd}
\usepackage[normalem]{ulem}

\usepackage{tikz}
\usetikzlibrary{mindmap,backgrounds, calc}
\usetikzlibrary{matrix,chains,scopes,positioning,arrows,fit}
\usetikzlibrary{positioning, shapes, patterns}
\usetikzlibrary{automata} 
\usetikzlibrary{shapes.geometric, arrows, arrows.meta}
\usetikzlibrary{positioning,decorations.markings}

\newtheorem{theorem}{Theorem}[section]
\newtheorem{lemma}[theorem]{Lemma}

\newtheorem{proposition}[theorem]{Proposition}

\theoremstyle{definition}
\newtheorem{definition}[theorem]{Definition}

\newtheorem{remark}[theorem]{Remark}
\numberwithin{equation}{section}

\newtheorem*{theorem0}{Theorem}
\newtheorem{theoremA}{Theorem} 
\newtheorem{theoremB}{Theorem}
\newtheorem{theoremC}{Theorem}

\newcommand{\N}{\mathbb{N}}
\newcommand{\K}{\mathbb{K}}

\newcommand{\Lin}{\mathcal{L}}

\DeclareMathOperator{\dist}{dist\,}

\DeclareMathOperator{\co}{co}
\DeclareMathOperator{\Id}{Id}

\DeclareMathOperator{\re}{Re}

\newcommand{\nn}[1]{{\left\vert\kern-0.25ex\left\vert\kern-0.25ex\left\vert #1 
		\right\vert\kern-0.25ex\right\vert\kern-0.25ex\right\vert}}

\renewcommand{\geq}{\geqslant}
\renewcommand{\leq}{\leqslant}

\newcommand{\norm}[1]{\left\Vert#1\right\Vert}

\newcommand{\NA}{\operatorname{NA}}

\newcommand{\spann}{\operatorname{span}}

\newcommand{\pten}{\ensuremath{\widehat{\otimes}_\pi}}

\newcommand{\eps}{\varepsilon}
\newcommand{\n}{\left\Vert\cdot\right\Vert}

\begin{document}
\setcounter{tocdepth}{1}


\title{On the existence of non-norm-attaining operators}

\author[Dantas]{Sheldon Dantas}
\address[Dantas]{Departament de Matemàtiques and Institut Universitari de Matemàtiques i Aplicacions de Castelló (IMAC), Universitat Jaume I, Campus del Riu Sec. s/n, 12071 Castelló, Spain. \newline
	\href{http://orcid.org/0000-0001-8117-3760}{ORCID: \texttt{0000-0001-8117-3760} } }
\email{\texttt{dantas@uji.es}}

\author[Jung]{Mingu Jung}
\address[Jung]{Department of Mathematics, POSTECH, Pohang 790-784, Republic of Korea \newline
	\href{http://orcid.org/0000-0000-0000-0000}{ORCID: \texttt{0000-0003-2240-2855} }}
\email{\texttt{jmingoo@postech.ac.kr}}

\author[Mart\'inez-Cervantes]{Gonzalo Mart\'inez-Cervantes}
\address[Mart\'inez-Cervantes]{Universidad de Murcia, Departamento de Matem\'{a}ticas, Campus de Espinardo 30100 Murcia, Spain
	\newline
	\href{http://orcid.org/0000-0002-5927-5215}{ORCID: \texttt{0000-0002-5927-5215} } }	

\email{gonzalo.martinez2@um.es}

\thanks{S. Dantas was supported by the project OPVVV CAAS CZ.02.1.01/0.0/0.0/16\_019/0000778 and by the Estonian Research Council grant PRG877. The second author was supported by NRF (NRF-2019R1A2C1003857). The third author was partially supported by \textit{Fundaci\'on S\'eneca} [20797/PI/18], \textit{Agencia Estatal de Investigaci\'on} [MTM2017-86182-P, cofunded by ERDF, EU] and by the European Social Fund (ESF) and the Youth European Initiative (YEI) under \textit{Fundaci\'on S\'eneca} [21319/PDGI/19].
	}



\begin{abstract} In this paper we provide necessary and sufficient conditions for the existence of non-norm-attaining operators in $\mathcal{L}(E, F)$. By using a theorem due to Pfitzner on James boundaries, we show that if there exists a relatively compact set $K$ of $\mathcal{L}(E, F)$ (in the weak operator topology) such that $0$ is an element of its closure (in the weak operator topology) but it is not in its norm closed convex hull, then we can guarantee the existence of an operator which does not attain its norm. This allows us to provide the following generalization of results due to Holub and Mujica. If $E$ is a reflexive space, $F$ is an arbitrary Banach space, and the pair $(E, F)$ has the bounded compact approximation property, then the following are equivalent:
\begin{itemize} 
\item[(i)]$\mathcal{K}(E, F) = \mathcal{L}(E, F)$;
\item[(ii)] Every operator from $E$ into $F$ attains its norm;
\item[(iii)] $(\mathcal{L}(E,F), \tau_c)^* = (\mathcal{L}(E, F), \n)^*$;
\end{itemize}
where $\tau_c$ denotes the topology of compact convergence. We conclude the paper presenting a characterization of the Schur property in terms of norm-attaining operators.
 
\end{abstract}
\maketitle


\section{Introduction}

The famous James theorem states that a Banach space $E$ is reflexive if and only if every bounded linear functional on $E$ attains its norm. By using this characterization, one can check that if every bounded linear operator from a Banach space $E$ into a Banach space $F$ is norm-attaining, then $E$ must be reflexive, whereas the range space $F$ is not forced to be reflexive in general. Indeed, every bounded linear operator from a reflexive space into a Banach space which satisfies the Schur property is compact and any compact operator from a reflexive space into an arbitrary Banach space is norm-attaining. Therefore, it seems natural to wonder whether it is possible to guarantee the existence of a non-norm-attaining operator from the existence of a non-compact operator. This brings us back to the 70's when J.R.~Holub proved that this is, in fact, true under approximation property assumptions (see \cite[Theorem 2]{H}). Almost thirty years later, J.~Mujica improved Holub's result by using the compact approximation property (see \cite[Theorem 2.1]{Mujica}), which is a weaker assumption than the approximation property. However, both results require the reflexivity on both domain and range spaces, so the following question arises naturally.

\begin{center}
	{\it Given a reflexive space $E$ and an arbitrary Banach space $F$, under which assumptions we may guarantee the existence of non-norm-attaining operators in $\mathcal{L}(E, F)$?}
\end{center}

Coming back to Holub and Mujica's results, we would like to highlight what they proved in the direction of the above question. For a background on necessary definitions and notations, we refer the reader to Section \ref{preliminaries}. In what follows, $\tau_c$ denotes the topology of compact convergence and $\|\cdot\|$ denotes the norm topology in $\mathcal{L}(E, F)$. 

\begin{theorem0}[\mbox{\cite[Theorem 2]{H} and \cite[Theorem 2.1]{Mujica}}] Let $E$ and $F$ be both reflexive spaces. 
	\begin{itemize}
		\setlength\itemsep{0.3em} 
		\item[(a)] If $\mathcal{L}(E,F)$ is non-reflexive, there is a non-norm-attaining operator $S \in \mathcal{L}(E,F)$.
		\item[(b)] If $E$ or $F$ has the (compact) approximation property, then the following statements are equivalent.
		\begin{itemize} 
				\setlength\itemsep{0.3em} 
		\item[(i)] There exists a non-norm-attaining operator $S \in \mathcal{L}(E,F)$;
		\item[(ii)] $\mathcal{L}(E,F)\neq \mathcal{K}(E,F)$;
		\item[(iii)] $\mathcal{L}(E,F)$ is non-reflexive;
		\item[(iv)] $(\mathcal{L}(E,F), \|\cdot\|)^* \neq (\mathcal{L} (E,F), \tau_c)^*$.
		\end{itemize} 
		\end{itemize}
\end{theorem0}
\noindent
The proof of the above result relies on the fact that if $F$ is a reflexive space, then $\mathcal{L}(E, F)$ is the dual space of the projective tensor product $E \pten F^*$. However, if the range space $F$ is non-reflexive, then $\mathcal{L}(E,F)$ is always non-reflexive (see, for instance, \cite{Ruckle}). 

As a way of extending the above results to the case of non-reflexive range spaces, we borrow some of the techniques used by R.C.~James (see \cite{J, J1}). As a matter of fact, one of his results \cite[Theorem 1]{J1} implies that a separable Banach space $E$ is non-reflexive if and only if given $0 < \theta < 1$, there exists a sequence $(x_n^*)$ in $B_{E^*}$ such that $x_n^* \stackrel{w^*}{\longrightarrow} 0$ and $\dist (0, \co \{ x_n^* : n \in \N \}) > \theta$, which in turn is equivalent to the existence of a relatively weak* compact set $K \subseteq B_{E^*}$ such that $0 \in \overline{K}^{w^*}$ and $0 \not\in \overline{\co}^{\norm{\cdot}}(K)$. This motivates us to define the following property.

\begin{definition} \label{definitionpropertystar} We say that a pair $(E, F)$ of Banach spaces has the {\it James property} if there exists a relatively WOT-compact set $K \subseteq \mathcal{L}(E, F)$ such that $0 \in \overline{K}^{WOT}$ and $0 \not\in \overline{\co}^{\norm{\cdot}}(K)$. 
\end{definition}

We will prove that the James property is a sufficient condition to guarantee the existence of an operator which does not attain its norm, which is our first aim in the present paper.

\begin{theoremA} \label{theoremA} Let $E$ and $F$ be Banach spaces. If the pair $(E, F)$ has the James property, then there exists a non-norm-attaining operator in $\mathcal{L}(E, F)$. 
\end{theoremA}

Next, we prove that $(\mathcal{L}(E,F), \|\cdot\|)^*$ does not coincide with $(\mathcal{L}(E,F),\tau_c)^*$ whenever a pair $(E, F)$ satisfies the James property (see Proposition \ref{duals_of_L(E,F)}). From this, we can see that whenever the pair $(E, F)$ has the James property, then the Banach space $\mathcal{L}(E, F)$ cannot be reflexive due to \cite[Lemma 2.3]{Mujica}.

We observe, for a reflexive space $E$ and an arbitrary Banach space $F$, that (1) the unit ball of $\mathcal{K}(E,F)$ is closed in the strong operator topology if and only if it is sequentially closed in this topology (see Lemma \ref{LemmaSequentialClosureOfK(E,F)}) and (2) $\mathcal{K}(E,F) = \mathcal{L}(E,F)$ implies that $(\mathcal{L}(E,F), \|\cdot\|)^* = (\mathcal{L}(E,F), \tau_c)^*$ by using the result \cite[Theorem 1]{FS} due to M. Feder and P. Saphar. 
Besides that, we consider the concept of the (bounded) compact approximation property for a pair of Banach spaces in the way as it is done in \cite{Bonde} and prove that $\mathcal{K}(E,F) = \mathcal{L} (E,F)$ when either (3) the norm-closed unit ball of $\mathcal{K}(E,F)$ is closed in the strong operator topology or (4) $(\mathcal{L}(E,F), \| \cdot\|)^* = (\mathcal{L}(E,F), \tau_c)^*$ under the just mentioned approximation property assumption (see Lemma \ref{LemmaGeneralizationMujica}). 
Combining (1)-(4) together with Theorem \ref{theoremA}, we get a generalization of Holub and Mujica's results, where we no longer need reflexivity on the target space $F$.

\begin{theoremB} \label{theoremB} Let $E$ be a reflexive space and $F$ be an arbitrary Banach space. Consider the following conditions.
	\begin{itemize}
	\setlength\itemsep{0.3em} 
		\item[(a)]$\mathcal{K}(E, F) = \mathcal{L}(E, F)$.
		\item[(b)] Every operator from $E$ into $F$ attains its norm.
		
		\item[(c)] The unit ball $B_{\mathcal{K}(E, F)}$ is closed in the strong operator topology.
		\item[(d)] $(\mathcal{L}(E,F), \tau_c)^* = (\mathcal{L}(E, F), \n)^*$.
	\end{itemize}
	Then, we always have (a) $\Longrightarrow$ (b) $\Longrightarrow$ (c) and (a) $\Longrightarrow$ (d) $\Longrightarrow$ (c). Additionally, if the pair $(E,F)$ has the bounded compact approximation property, then (c) $\Longrightarrow$ (a) and therefore all the statements are equivalent. 
\end{theoremB}

The following diagram summarizes most of the results included in this article. In what follows, $E$ is supposed to be any reflexive space and $F$ is any arbitrary Banach space. Moreover, BCAP stands for the bounded compact approximation property for the pair $(E, F)$ (see Definition \ref{BCAP} in Section \ref{preliminaries} below).

\vspace{0.5cm}

\begin{center}
	\begin{tikzpicture}[scale=0.95, baseline=0cm]
	\tikzstyle{caixa} = [rectangle, rounded corners, minimum width=.7cm, minimum height=.7cm, text centered, draw=black]
	\tikzstyle{dfletxa} = [double equal sign distance,-implies, shorten >= 2pt, , shorten <= 2pt]
	\tikzstyle{dfletxadotted} = [double equal sign distance, dashed, -implies, shorten >= 2pt, , shorten <= 2pt]
	
	\node (A) at (0,0) [caixa] {$\scriptstyle 
		\overline{B_{\mathcal{K}(E, F)}}^{SOT} \not= B_{\mathcal{K}(E, F)}
		$};
	\node (B) at (0,-1.5) [caixa] {$\scriptstyle 
		(E, F) \ \text{has the James property}
		$};
	\node (C) at (0,-3) [caixa] {$\scriptstyle 
		\mathcal{L}(E, F) \not= \mathrm{NA}(E, F) 
		$};
	\node (D) at (-6,-3) [caixa] {$\scriptstyle 
		\mathcal{L}(E, F) \not= \mathcal{K}(E, F)
		$};
	\node (E) at (6,-1.5) [caixa] {$\scriptstyle 
		(\mathcal{L}(E, F), \|\cdot\|)^* \not= (\mathcal{L}(E, F), \tau_c)^*
		$};
	\node (F) at (6,-3) [caixa] {$\scriptstyle 
		\mathcal{L}(E, F) \ \text{is non-reflexive}
		$};
	
	\draw [dfletxa] (A) -- (B);
	\draw [dfletxa] (B) -- (C);
	\draw [dfletxa] (C) -- (D);
	\draw [dfletxa] (B) -- (E);
	\draw [dfletxa] (E) -- (F);
	
	%
	
	\draw[dfletxadotted, ->, >=implies, rounded corners] (D.north)
	-- (-6,-2.0)
	-- node[right, pos=1.1] {$ \scriptstyle \text{If } (E,F) \text{ has the BCAP}$} (-6,0)
	-- (A.west);
	
	\draw[dfletxadotted, rounded corners] (F.south)
	-- (6,-4)
	-- node[above, pos=.8] {$\scriptstyle \text{ If } F \text{ is reflexive}$} (2.5,-4)
	-- (0,-4)
	-- (C.south);
	\end{tikzpicture}
\end{center}

\vspace{0.5cm}

Finally, as an application of Theorem \ref{theoremB}, we connect the Schur property with the case where every operator attains its norm, and obtain the following characterization (see Theorem \ref{charcterizations_Schur}).

\begin{theoremC} \label{theoremC} Let $F$ be a Banach space. The following statements are equivalent.
	
	\begin{itemize}
		\setlength\itemsep{0.3em} 
		\item[(a)] $F$ has the Schur property.
		\item[(b)] $\mathcal{K}(E, F) = \mathcal{L}(E, F)$ for every reflexive space $E$.
		\item[(c)] $\NA(E, F) = \mathcal{L}(E, F)$ for every reflexive space $E$.
		\item[(d)] $\mathcal{K} (G, F) = \mathcal{L} (G, F)$ for every reflexive space $G$ with basis.
		\item[(e)] 		$\NA (G,F) = \mathcal{L} (G,F)$ for every reflexive space $G$ with basis.
	\end{itemize}
\end{theoremC}

\section{Preliminaries} \label{preliminaries}

Throughout the paper, $E$ and $F$ will be Banach spaces over {a field} $\K$, which can be either the real or complex numbers. We denote by $B_E$ and $S_E$ the closed unit ball and the unit sphere of the Banach space $E$, respectively. For a subset $K$ of $E$, $\co (K)$ (resp., $\overline{\co} (K)$) denotes the convex hull (resp., closed convex hull) of $K$. 
The space of all bounded linear operators from $E$ into $F$ is denoted by $\mathcal{L}(E, F)$. The symbol $\mathcal{K}(E, F)$ (resp., $\mathcal{W}(E,F)$) stands for the space of all compact operators (resp., weakly compact operators) from $E$ into $F$, whereas the symbol $\mathcal{F}(E, F)$ is used to denote the space of all finite-rank operators. Recall that $T \in \mathcal{L} (E,F)$ is completely continuous if $T$ sends weakly null sequences in $E$ to norm null sequences in $F$. We denote by $\mathcal{V}(E,F)$ the space of all completely continuous operators from $E$ into $F$. 
Let us denote by $\mathcal{W}_\infty (E,F)$ the space of all weakly $\infty$-compact operators from $E$ into $F$, which are introduced in \cite{SK}. A subset $C$ of a Banach space $E$ is called relatively weakly $\infty$-compact if there exists a weakly null sequence $(x_n)$ in $E$ such that $C \subset \{ \sum_{n=1}^\infty a_n x_n : (a_n) \in B_{\ell_1} \}$ and an operator $T \in \Lin (E,F)$ is said to be weakly $\infty$-compact if $T(B_E)$ is a relatively weakly $\infty$-compact subset of $F$.
Finally, let us recall that an operator $T \in \mathcal{L}(E, F)$ attains its norm or it is norm-attaining if there exists $x \in B_E$ such that $\|T(x)\| = \|T\|$. By $\NA(E,F)$ we mean the set of all norm-attaining operators from $E$ into $F$. If $E = F$, then we simply write $\NA(E)$ instead of $\NA(E, E)$ and we do the same with the above classes of operators.


{We will be using} different topologies in $\mathcal{L}(E, F)$. We denote by $\tau_c$ the topology of compact convergence, i.e.~the topology of uniform convergence on compacts subsets of $E$. The weak operator topology (WOT, for short) is defined by the basic neighborhoods 
\begin{equation*}
N(T; A, B, \eps) := \big\{ S \in \mathcal{L}(E, F): |y^*(T - S)(x)| < \eps, \mbox{ for every } y^* \in B, x \in A \big\},
\end{equation*}
where $A$ and $B$ are arbitrary finite sets in $E$ and $F^*$, respectively, and $\eps > 0$. Thus, in the WOT, a net $(T_{\alpha})$ converges to $T$ if and only if $(y^*(T_{\alpha}(x)))$ converges to $y^*{(T(x))}$ for every $x \in E$ and $y^* \in F^*$. Analogously, the strong operator topology (SOT, for short) is defined by the basic neighborhoods 
\begin{equation*}
N(T; A, \eps) := \big\{ S \in \mathcal{L}(E, F): \|(T - S)(x)\| < \eps, \mbox{ for every } x \in A \big\},
\end{equation*}
where $A$ is an arbitrary finite set in $E$ and $\eps > 0$. Thus, a net $(T_{\alpha})$ converges in the SOT to $T$ if and only if $(T_{\alpha}(x))$ converges in norm to $T(x)$ for every $x \in E$. We will deal with SOT and WOT closures of bounded sets in $\mathcal{L}(E,F)$. It is worth mentioning that, for a bounded convex set in $\mathcal{L}(E,F)$, the WOT closure and the SOT closure coincide \cite[Corollary VI.1.5]{DS}. Thus, the SOT and the WOT in some results in this paper can be interchanged. For a more detailed exposition on topologies in $\mathcal{L}(E, F)$, we refer the reader to \cite{DF, DS}.

 {Let us present now the necessary definitions on approximation properties we will need}. A Banach space $E$ is said to have the {approximation property (AP, for short)} if the identity operator $\Id_E$ in 
$\mathcal{L}(E)$ belongs to $\overline{ \mathcal{F} (E)}^{\tau_c}$. Given $\lambda \geq 1$, $E$ is said to have the $\lambda$-bounded approximation property ($\lambda$-BAP, for short) when $\Id_E$ belongs to $\lambda \overline{B_{\mathcal{F}(E)}}^{\tau_c}$. A Banach space is said to have {the} bounded approximation property (BAP, for short) if it has the $\lambda$-BAP for some $\lambda \geq 1$. Especially, when $\lambda = 1$, we say that $E$ has the metric approximation property (MAP, for short). Also, recall that $E$ is said to have the compact approximation property (CAP, for short) if the identity operator $\Id_E$ in 
$\mathcal{L}(E)$ belongs to $\overline{ \mathcal{K} (E)}^{\tau_c}$. 
The $\lambda$-bounded compact approximation property ($\lambda$-BCAP, for short), bounded compact approximation property (BCAP, for short) and metric compact approximation property (MCAP, for short) for a Banach space $E$ can be defined in an analogous way. 
It is known that a reflexive space has the AP if and only if it has the MAP (see \cite{Gro}). Analogously, every reflexive space with the CAP also has the MCAP (see \cite[Proposition 1 and Remark 1]{CJ}).
We refer the reader to \cite{Linden-Tz, LT} and \cite{Casazza} for background.

On the other hand, E.~Bonde introduced in \cite{Bonde} the AP and $\lambda$-BAP for pairs of Banach spaces, that is, a pair $(E,F)$ of Banach spaces is said to have {the} AP (resp., $\lambda$-BAP) if any operator $T \in \mathcal{L} (E,F)$ belongs to $\overline{ \mathcal{F} (E,F)}^{\tau_c}$ (resp., $\lambda \overline{B_{\mathcal{F}(E,F)}}^{\tau_c}$ for some $\lambda \geq 1$). It is clear that if $E$ or $F$ has the AP (resp., BAP), then the pair $(E,F)$ has the AP (resp., BAP). It is observed in \cite[Section 4]{Bonde} that there are pairs of Banach spaces $(E,F)$ with the BAP such that $E$ and $F$ do not have the BAP. 
Similarly, we have the following.

\begin{definition} \label{BCAP} The pair $(E,F)$ of Banach spaces is said to have the {compact approximation property (CAP, for short)} (resp., {bounded compact approximation property (BCAP, for short))} if any operator $T \in \mathcal{L} (E,F)$ belongs to $\overline{\mathcal{K} (E,F)}^{\tau_c}$ (resp., $\lambda \overline{B_{\mathcal{K}(E,F)}}^{\tau_c}$ for some $\lambda \geq 1$). In the case when $\lambda =1$, we say that {the pair} $(E,F)$ has the {metric compact approximation property (MCAP, for short)}.
\end{definition}

As a matter of fact, \cite[Example 4.2]{Bonde} shows that there are Banach spaces $E$ and $F$ such that $(E,F)$ has the BCAP while $E$ and $F$ do not have the CAP. Thus, assuming that a pair $(E,F)$ of Banach spaces has CAP is more general than $E$ or $F$ has the CAP. {We will be using this fact without any explicit reference throughout the paper.}





\section{The Results} \label{sectionsuffientconditions}

In this section, we shall prove Theorems \ref{theoremA}, \ref{theoremB}, \ref{theoremC}, and their consequences. We start by proving Theorem \ref{theoremA}. To do so, let us recall that a set $B \subset B_{E^*}$ is called a James boundary of a Banach space $E$ if for every $x \in S_E$, there exists $f \in B$ such that $f(x) = 1$. For a subset $G$ of $E^*$, we shall denote by $w(E, G)$ the weak topology of $X$ induced by $G$.

\begin{proof}[Proof of Theorem \ref{theoremA}] Let us assume by contradiction that every operator from $E$ into $F$ attains its norm. Then, the family 
\begin{equation*} 
B:= \Big\{ x \otimes y^*: x \in S_E, y^* \in S_{F^*} \Big\}
\end{equation*} 
is a James boundary of $\mathcal{L}(E, F)$. Indeed, for an arbitrary operator $T \in \mathcal{L}(E, F) = \NA(E, F)$, take $x \in S_E$ to be such that $\|T(x)\| = \|T\|$ and then $y^* \in S_{F^*}$ to be such that $|y^*(T(x))| = \|T(x)\| = \|T\|$. Now, since $(E, F)$ has the James property, there exists a relatively WOT-compact set $K \subset \mathcal{L}(E, F)$ such that $0 \in \overline{K}^{WOT}$ and $0 \not\in \overline{\co}^{\norm{\cdot}}(K)$. By the Uniform Boundedness principle, the set $\overline{K}^{WOT}$ is norm-bounded. By hypothesis, $\overline{K}^{WOT}$ is WOT-compact or, equivalently, $w(\mathcal{L}(E, F), B)$-compact. By a theorem of Pfitzner (see \cite{Pf} or \cite[Theorem 3.121]{FHHMZ}), we have that $\overline{K}^{WOT}$ is weakly compact. Therefore, $0 \in \overline{K}^{WOT} = \overline{K}^w$, which in particular gives that $0 \in \overline{\co}^w(K) = \overline{\co}^{\norm{\cdot}}(K)$. This contradiction yields a non-norm-attaining operator $T \in \mathcal{L}(E, F)$ as desired.
\end{proof}

Let us observe that if a pair $(E,F)$ of Banach spaces has the James property, then the dual of $\Lin(E, F)$ endowed with the norm topology does not coincide with the dual of $\Lin(E, F)$ endowed with the topology $\tau_c$ of compact convergence. As a matter of fact, if $K$ is a subset of $E$ given as in Definition \ref{definitionpropertystar}, then there exists $\varphi \in (\mathcal{L}(E, F), \|\cdot\|)^*$ such that $0 = \re \varphi(0) > \sup \left\{ \re \varphi(T): T \in \overline{\co}(K) \right\}$ thanks to the Hahn-Banach separation theorem. This implies that $\varphi$ cannot be in $(\mathcal{L}(E, F), \tau_c)^*$ since $0 \in \overline{\co}^{WOT}(K) = \overline{\co}^{\tau_c}(K)$. Moreover, using \cite[Lemma 2.3]{Mujica}, we see that if $(\Lin (E,F), \| \cdot \|)^* \neq (\Lin (E,F), \tau_c)^*$, then the space $\Lin (E,F)$ cannot be reflexive. Summarizing, we obtain the following result.

\begin{proposition} \label{duals_of_L(E,F)} Let $E$ and $F$ be Banach spaces. If the pair $(E, F)$ has the James property, then 
\begin{enumerate}
	\setlength\itemsep{0.3em} 
	\item[(i)] $(\mathcal{L}(E, F), \|\cdot\|)^* \not= (\mathcal{L}(E, F), \tau_c)^*$.
	\item[(ii)] $\mathcal{L}(E, F)$ is non-reflexive. 
	\end{enumerate} 
\end{proposition}




One easy consequence of Theorem \ref{theoremA} is that if $E$ is reflexive and a pair $(E,F)$ has the James property, then $\mathcal{K}(E,F)$ cannot be equal to the {whole} space $\mathcal{L} (E,F)$. As a matter of fact, the following result gives us a rather general observation.

\begin{proposition} \label{holubtypetheoremfirstpart} Let $E$ be a reflexive space and $F$ be an arbitrary Banach space. If $\mathcal{K}(E, F) = \mathcal{L}(E, F)$, then $(\mathcal{L}(E, F), \|\cdot\|)^* = (\mathcal{L}(E, F), \tau_c)^*$.
\end{proposition}

\begin{proof}
	Let $D: E \pten F^* \longrightarrow (\mathcal{L}(E, F), \tau_c)^*$ be defined by
$D (z)(T):= \sum_{n=1}^{\infty} y_n^*(T(x_n))$
for every $z \in E \pten F^*$ with $z = \sum_{n=1}^{\infty} x_n \otimes y_n^*$ and $T \in \mathcal{L}(E, F)$. It is well known that $D$ is a surjective map (see, for example, \cite[5.5, pg. 62]{DF}). Therefore, we have that
$(\mathcal{L}(E, F), \tau_c)^* =( E \pten F^* )/ \ker D.$
On the other hand, from the result \cite[Theorem 1]{FS}, we have that the map $V: E \pten F^* \longrightarrow (\mathcal{K}(E, F), \|\cdot\|)^*$ defined by
$V(z)(T) := \sum_{n=1}^{\infty} y_n^*(T(x_n))$
for $z = \sum_{n=1}^{\infty} x_n \otimes y_n^*$ and $T \in \mathcal{K}(E, F)$, satisfies the following: for every $\varphi \in (\mathcal{K}(E, F), \|\cdot\|)^*$, there exists $v \in E \pten F^*$ such that $\varphi = V(v)$ and $\|\varphi\| = \|v\|$. In particular, we have that $(\mathcal{K}(E, F), \|\cdot\|)^* = (E \pten F^*) / \ker V$. 
Thus, if $\mathcal{K}(E, F) = \mathcal{L}(E, F)$, then $D(z)(T) = V(z)(T)$ for every $z \in E \pten F^*$ and every $T \in \mathcal{K}(E, F)$; hence $\ker D = \ker V$ and $(\Lin(E,F), \| \cdot\|)^* =(\mathcal{L}(E, F), \tau_c)^*$. 
\end{proof}

Let us now go towards the proof of Theorem \ref{theoremB}. We show the following result which will help us to prove that if $(E, F)$ does not satisfy the James property, then $\overline{B_{\mathcal{K}(E, F)}}^{SOT}$ coincides with $B_{\mathcal{K}(E, F)}$. Recall that the sequential closure of a set in a topological space is the family of all limit points of sequences on the set in consideration. 


\begin{lemma} \label{oldcharacterizationJamesProperty} Let $E$ and $F$ be Banach spaces. Suppose that there exists a norm-closed convex set $C \subseteq \mathcal{L}(E, F)$ which is not sequentially closed in the strong operator topology. Then, $(E, F)$ has the James property. 
\end{lemma}

\begin{proof}  Suppose that $C \subseteq \mathcal{L}(E, F)$ is norm-closed but not SOT-sequentially closed. This implies that there exists a sequence of operators $(R_n) \subseteq C$ such that $(R_n)$ converges in the SOT (and therefore in the WOT) to an operator $R \notin C$. We may (and we do) suppose that $R = 0$. Set $K:= \{R_n:n \in \N\} \subset \mathcal{L}(E, F)$. Therefore, $K$ is relatively WOT-compact, $0 \in \overline{K}^{WOT}$ but $0$ cannot be in $\overline{\co}(K)$ by hypothesis. Therefore, $(E, F)$ has the James property.
\end{proof}

It is not difficult to check that, {for} a bounded subset $C$ of $\Lin (E,F)$, with $E$ separable, the SOT-closure of $C$ coincides with the SOT-sequential closure of $C$. Furthermore, the following result shows that the unit ball of $\mathcal{K} (E,F)$ is SOT-closed if it is SOT-sequentially closed under the assumption that $E$ has the separable complementation property. Recall that a Banach space $E$ is said to have the separable complementation property if for every separable subspace $Y$ in $E$, there is a separable subspace $Z$ with $Y \subset Z \subset E$ and $Z$ is complemented in $E$. It is worth mentioning that D.~Amir and J.~Lindenstrauss proved in \cite{AL} that every weakly compactly generated Banach space (and therefore every reflexive space) has the separable complementation property.

\begin{lemma}\label{LemmaSequentialClosureOfK(E,F)} Let $E$ be a Banach space with the separable complementation property and $F$ be an arbitrary Banach space. Then, the unit ball $B_{\mathcal{K}(E,F)}$ is SOT-closed if and only if it is SOT-sequentially closed.
\end{lemma}

\begin{proof}
Is it enough to check that if $B_{\mathcal{K}(E,F)}$ is not SOT-closed then it is not SOT-sequentially closed. 
Suppose that $T$ is an operator which belongs to the SOT-closure of $B_{\mathcal{K}(E,F)}$ but not to $B_{\mathcal{K}(E,F)}$. Notice that $T$ is {non-compact}; hence there exists a separable subspace $E_0$ of $E$ such that $T|_{E_0}$ is non-compact. 
Choose a separable subspace $Z$ of $E$ such that $E_0 \subset Z \subset E$ and $Z$ is complemented in $E$. 
Notice that $T|_{Z}$ is non-compact and belongs to the SOT-closure of $B_{\mathcal{K} (Z, F)}$. 
As $Z$ is separable, we have that $T|_Z$ is indeed in the SOT-sequential closure of $B_{\mathcal{K} (Z,F)}$. Let $(K_n)$ be a sequence in $B_{\mathcal{K} (Z,F)}$ converging to $T|_Z$ in the SOT. Letting $P$ be a projection from $E$ onto $Z$, it is immediate that $T|_Z \circ P$ is non-compact and $K_n \circ P$ is SOT-convergent to $T|_Z \circ P$. This proves that $B_{\mathcal{K}(E,F)}$ is not SOT-sequentially closed.
\end{proof}

It is worth mentioning however that the SOT-closure and SOT-sequential closure are different in general  as the following remark shows.

\begin{remark} \label{sequentialclosure} 
In general, it is not true that the SOT-sequential closure of a bounded convex set $C$ in $\mathcal{L}(E,F)$ coincides with the SOT-closure of $C$.
An example is given by 
\begin{equation*} 
C := \Big\{ T \in B_{\mathcal{L}(\ell_2(\omega_1))} \colon \mbox{there is} \ \alpha<\omega_1 \ \mbox{such that} \left(T(x)\right)_{\beta}=0 \ \mbox{for every} \ \beta>\alpha,~x\in \ell_2(\omega_1) \Big\}.
\end{equation*} 
It is immediate that $C$ is SOT-sequentially closed. Nevertheless, since the canonical projections $P_\alpha \in \mathcal{L}(\ell_2(\omega_1))$ with $\alpha<\omega_1$, defined by $\left(P_\alpha(x)\right)_\beta=x_\beta$ if $\beta\leq \alpha$ and $0$ otherwise, are in $C$ and satisfy that $\{P_\alpha\}_{\alpha<\omega_1}$ SOT-converges to the identity, which is not in $C$, it follows that $C$ is not SOT-closed.
\end{remark}

Notice that if $E$ is reflexive, then it has the separable complementation property. By Lemma \ref{LemmaSequentialClosureOfK(E,F)}, $B_{\mathcal{K}(E, F)}$ is SOT-closed if and only if it is SOT-sequentially closed. Therefore, if we assume that $\overline{B_{\mathcal{K}(E, F)}}^{SOT} \not= B_{\mathcal{K}(E, F)}$, then $(E, F)$ has the James property by Lemma \ref{oldcharacterizationJamesProperty}. Therefore, we have the following result.

\begin{proposition} \label{diagramfirstimplication} Let $E$ and $F$ be Banach spaces. If $\overline{B_{\mathcal{K}(E, F)}}^{SOT} \not= B_{\mathcal{K}(E, F)}$, then $(E, F)$ has the James property. 
\end{proposition}

In order to prove Theorem \ref{theoremB}, we need the following lemma.

\begin{lemma}
\label{LemmaGeneralizationMujica}
Let $E$ be a reflexive space and $F$ be an arbitrary Banach space. Suppose that the pair $(E,F)$ has the BCAP. 
Then $\mathcal{K} (E,F) = \mathcal{L}(E,F)$ if and only if the unit ball $B_{\mathcal{K}(E,F)}$ is SOT-closed. 
\end{lemma}


\begin{proof}
First, note that as the pair $(E,F)$ has the BCAP, 
$\mathcal{L}(E,F)=\bigcup_{\lambda>0} \lambda\overline{ B_{\mathcal{K}(E,F)}}^{\tau_c}$.
Since $\overline{ B_{\mathcal{K}(E,F)}}^{\tau_c} \subseteq \overline{ B_{\mathcal{K}(E,F)}}^{SOT}$, we have that 
$\mathcal{L}(E,F)=\bigcup_{\lambda>0} \lambda\overline{ B_{\mathcal{K}(E,F)}}^{SOT}.$
So, if we assume $B_{\mathcal{K}(E,F)}$ to be SOT-closed, then 
\begin{equation*} 
\mathcal{L}(E,F)= \bigcup_{\lambda>0} \lambda\overline{ B_{\mathcal{K}(E,F)}}^{SOT}=\bigcup_{\lambda>0} \lambda B_{\mathcal{K}(E,F)} =\mathcal{K}(E,F).
\end{equation*} 
The other implication is immediate.
\end{proof}

To prove Theorem \ref{theoremB}, we consider the following conditions and we use Theorem \ref{theoremA}, Proposition \ref{duals_of_L(E,F)}, Proposition \ref{holubtypetheoremfirstpart}, Proposition \ref{diagramfirstimplication}, and Lemma \ref{LemmaGeneralizationMujica}.

	\begin{itemize}
	\setlength\itemsep{0.3em} 
	\item[(a)]$\mathcal{K}(E, F) = \mathcal{L}(E, F)$.
	\item[(b)] Every operator from $E$ into $F$ attains its norm.
	
	\item[(c)] The unit ball $B_{\mathcal{K}(E, F)}$ is closed in the strong operator topology.
	\item[(d)] $(\mathcal{L}(E,F), \tau_c)^* = (\mathcal{L}(E, F), \n)^*$.
\end{itemize}

\begin{proof}[Proof of Theorem \ref{theoremB}] Let $E$ be reflexive and $F$ be an arbitrary Banach space.
It is clear that (a) $\Longrightarrow$ (b). Moreover, (b) implies that $(E,F)$ does not have the James property (by applying Theorem \ref{theoremA}), which in turn implies (c) (by applying Proposition \ref{diagramfirstimplication}). On the other hand, Proposition \ref{holubtypetheoremfirstpart} shows (a) $\Longrightarrow$ (d). By Proposition \ref{duals_of_L(E,F)}, (d) implies that $(E,F)$ does not have the James Property and, therefore, it implies (c) (by applying Proposition \ref{diagramfirstimplication}). Finally, if the pair $(E,F)$ has the BCAP, then the implication (c) $\Longrightarrow$ (a) follows from Lemma \ref{LemmaGeneralizationMujica}. 
\end{proof}

M.I. Ostrovskii asked in \cite[\textsection12, pg.~65]{MP} whether there exist infinite dimensional Banach spaces on which every operator attains its norm (this question is also asked in \cite[Problem 8]{KOS} and \cite[Problem 217]{GMZ}). By Holub's Theorem, if such an infinite dimensional Banach space exists, it cannot have the AP. Theorem \ref{THEOFINALGENERALIZATION} below should be seen as a generalization of this fact.   Let us recall that given a (norm-closed) operator ideal $\mathcal{A}$ and $\lambda \geq 1$, a Banach space $E$ is said to have the $\lambda$-$\mathcal{A}$-approximation property (for short, $\lambda$-$\mathcal{A}$-AP) if the identity operator $\Id_E$ belongs to $\overline{\{ T \in \mathcal{A} (E,E) : \|T \| \leq \lambda \}}^{\tau_c}$. We say that $E$ has the bounded-$\mathcal{A}$-AP if it has the $\lambda$-$\mathcal{A}$-AP for some $\lambda \geq 1$. This general approximation property has been studied, for instance, in \cite{GW, LO, O1, Reinov}.

\begin{theorem}
	\label{THEOFINALGENERALIZATION}
	If there is an infinite dimensional Banach space $E$ such that every operator on $\mathcal{L}(E)$ attains its norm, then $E$ does not have the bounded $\mathcal{A}$-approximation property for any  ideal $\mathcal{A}$ not containing the identity on $E$.
\end{theorem}
\begin{proof}
As it is highlighted in \cite[\textsection12, pg. 66]{MP}, due to a result of N.J. Kalton, if such a Banach space $E$ exists, then it must be separable. Therefore, the SOT-closure of the set $B= \{T \in \mathcal{A} (E,E) : \|T \| \leq 1\}$ in $\Lin(E)$ coincides with its SOT-sequential closure. Thus, if every operator on $\Lin(E)$ attains its norm, then $B$ is SOT-closed by Lemma \ref{oldcharacterizationJamesProperty}.
Suppose that $E$ has the bounded $\mathcal{A}$-approximation property. Then, since $\overline{B}^{\tau_c} \subset \overline{B}^{SOT}=B $, we have that $B$ contains a multiple of the identity and therefore  $\mathcal{A}$ contains the identity on $E$.
\end{proof}

Let us conclude the paper by showing the proof of Theorem \ref{theoremC} as a direct implication of Theorem \ref{theoremB} and Proposition \ref{main_tool_for_theorem_C} below. Recall that a Banach space $E$ has the {\it Schur property} if every weakly convergent sequence is norm convergent. It is known that a Banach space $F$ has the Schur property if and only if every weakly compact operator from $E$ into $F$ is compact for any Banach space $E$ (see, for example, \cite[3.2.3, pg. 61]{Pietsch}). Also, it is proved in \cite[Theorem 1]{DFLORT} that a Banach space $F$ has the Schur property if and only if the weak Grothendieck compactness principle holds in $F$, that is, every weakly compact subset of $F$ is contained in the closed convex hull of a weakly null sequence. Afterwards, W.B. Johnson et al., gave an alternative proof in \cite[Theorem 1.1]{JLO} for this result by using the Davis-Figiel-Johnsonn-Pe{\l}czy\'nski factorization theorem \cite{DFJP}. Moreover, it is observed in \cite[Theorem 3.3]{JLO} that a Banach space $F$ has the Schur property if and only if $\mathcal{W}_\infty (E, F) \subset \mathcal{W} (E,F)$ for every Banach space $E$.

The following result will be used as an important tool afterwards. 

\begin{proposition}\label{main_tool_for_theorem_C}
Let $F$ be a Banach space. If $F$ fails to have the Schur property, then there exists a reflexive space with basis $E$ such that $\mathcal{K}(E,F) \neq \mathcal{L}(E,F)$. 
\end{proposition} 

\begin{proof}
Take $(x_n) \subseteq S_F$ to be a weakly null sequence in $F$, which is not norm null. Since the absolute closed convex hull of $\{x_n: n \in \N\}$ is weakly compact, the operator $T \in \mathcal{L}(\ell_1 ,F)$ given by $T (e_n) := x_n$ for each $n \in \N$ defines a weakly compact operator (which is not compact). By Davis-Figiel-Johnsonn-Pe{\l}czy\'nski factorization theorem \cite{DFJP}, there exists a reflexive space $E_0$ such that $T= S \circ R$, where $R \in \Lin(\ell_1, E_0)$ and $S \in \Lin(E_0, F)$. In particular, notice that $S$ cannot be a compact operator. Now, pick a weakly null sequence $(v_n) \subseteq E_0$ so that $S(v_n)$ does not admit a convergent subsequence. Since $(v_n)$ is weakly null, consider a subsequence which is a basic sequence of $E_0$ (see \cite[Proposition 1.5.4]{Albiac-Kalton}) and denote it again by $(v_n)$. Let $E := \overline{\spann} \{v_n\}_{n \in \N}$. Then, $E$ is a closed reflexive space with basis and $S(v_n)$ does not admit a convergence sequence. So, we conclude that $\mathcal{K}(E, F) \not= \mathcal{L}(E, F)$. 
\end{proof}

Compared to the previously known results in \cite{JLO}, Theorem \ref{charcterizations_Schur} below not only provides a new characterization of the Schur property in terms of norm-attaining operators, but also shows that we can restrict the possible candidates for a domain space as in the below items (f)-(i) by considering only reflexive Banach spaces {\it with basis}.  We refer the reader to Section \ref{preliminaries} for the definitions of the sets $\mathcal{V} (E,F)$ and $\mathcal{W}_\infty(E,F)$. It is immediate to notice that Theorem \ref{theoremC} follows from Theorem \ref{charcterizations_Schur}.

\begin{theorem}\label{charcterizations_Schur}
 Let $F$ be a Banach space. The following statements are equivalent.	
	\begin{itemize}
		\setlength\itemsep{0.3em} 
		\item[(a)] $F$ has the Schur property.
		\item[(b)] $\mathcal{K}(E, F) = \mathcal{L}(E, F)$ for every reflexive space $E$.
		\item[(c)] $\mathcal{W}_\infty (E,F) = \mathcal{L} (E,F)$ for every reflexive space $E$. 
		\item[(d)] $\mathcal{V} (E,F) = \mathcal{L}(E,F)$ for every reflexive space $E$.
		\item[(e)] $\NA(E, F) = \mathcal{L}(E, F)$ for every reflexive space $E$.
		\item[(f)] $\mathcal{K} (G, F) = \mathcal{L} (G, F)$ for every reflexive space $G$ with basis.
		\item[(g)] $\NA (G,F) = \mathcal{L} (G,F)$ for every reflexive space $G$ with basis.
		\item[(h)] $\mathcal{W}_\infty (G, F) = \mathcal{L} (G, F)$ for every reflexive space $G$ with basis.
		\item[(i)] $\mathcal{V}(G,F) = \mathcal{L} (G, F)$ for every reflexive space $G$ with basis.
	\end{itemize}
\end{theorem} 

\begin{proof} The following diagram holds. 

\vspace{0.2cm}	
	
\[
\begin{tikzcd}
                             & (c) \arrow[d, Rightarrow] \arrow[rd, Rightarrow] &              \\
(b) \arrow[ru, Rightarrow] \arrow[d, no head, Leftarrow] & (d) \arrow[d, Rightarrow]            & (h) \arrow[d, Rightarrow] \\
(a)                 & (e) \arrow[d, Rightarrow]            & (i) \arrow[ld, Rightarrow] \\
(f) \arrow[r, no head, Leftrightarrow]  \arrow[u, Rightarrow]           & (g)                       &              
\end{tikzcd}
\]

\vspace{0.2cm}

\noindent
Indeed, by definition we have that $\mathcal{K}(E,F) \subset \mathcal{W}_\infty (E,F) \subset \mathcal{W} (E,F)$ for any Banach space $E$ and $F$, and it is also known that $\mathcal{K}(E,F) \subset \mathcal{W}_\infty (E,F) \subset \mathcal{V} (E,F)$ (see \cite[Proposition 3.1]{JLO}). Moreover, if $T$ is an element of $\mathcal{V}(E,F)$ with $E$ reflexive, then $T \in \NA (E,F)$ thanks to the weak sequential compactness of $B_E$. 
Thus, it is immediate that (a) $\Longrightarrow$ (b) $\Longrightarrow$ (c) $\Longrightarrow$ (d) $\Longrightarrow$ (e) $\Longrightarrow$ (g) and (c) $\Longrightarrow$ (h) $\Longrightarrow$ (i) $\Longrightarrow$ (g) hold. As a reflexive Banach space with basis has the MAP, (f) $\Longleftrightarrow$ (g) follows from Theorem \ref{theoremB}. Finally, (f) $\Longrightarrow$ (a) is already obtained by Proposition \ref{main_tool_for_theorem_C}. 
\end{proof}

\bigskip

\noindent
\textbf{Acknowledgements.} We would like to thank Jos\'e Rodr\'iguez for suggesting us Definition \ref{definitionpropertystar} and the use of Pfitzner's Theorem to strengthen and simplify part of the content of the paper. We are also grateful to Richard Aron, Gilles Godefroy, Manuel Maestre, and Miguel Mart\'in for fruitful conversations on the topic of the present paper.

\end{document}